\documentclass[12pt]{amsart}
\usepackage{amsmath, amssymb, amscd, amsthm, verbatim,enumerate,a4wide}
\usepackage{hyperref}
\usepackage{color}

\numberwithin{equation}{section}
\theoremstyle{plain}
\newtheorem{theorem}{Theorem}[section]
\newtheorem{prop}[theorem]{Proposition}

\newtheorem{lemma}[theorem]{Lemma}

\theoremstyle{remark}
\newtheorem{remark}[theorem]{Remark}

\newcommand{\R}{\mathbb R}
\newcommand{\N}{\mathbb N}
\newcommand{\C}{\mathbb C}

\newcommand{\al}{\alpha}
\newcommand{\be}{\beta}

\newcommand{\Ga}{\Gamma}
\newcommand{\de}{\delta}
\newcommand{\De}{\Delta}

\newcommand{\si}{\sigma}
\newcommand{\te}{\theta}

\newcommand{\la}{\lambda}
\newcommand{\La}{\Lambda}

\newcommand{\Om}{\Omega}
\newcommand{\tensor}{\otimes}
\newcommand{\rFs}[5]{\,_{#1}F_{#2}\!\left( \genfrac{.}{.}{0pt}{}{#3}{#4}
\,;#5 \right)}

\newcommand{\su}{\mathfrak{su}}

\newcommand{\mhyphen}{\text{--}}

\title[Duality functions from Lie algebra representation]{Orthogonal stochastic duality functions from Lie algebra representations}
\author{Wolter Groenevelt}
\address{Technische Universiteit Delft, DIAM, PO Box 5031,
2600 GA Delft, the Netherlands}
\email{w.g.m.groenevelt@tudelft.nl}

\begin{document}
\maketitle

\begin{abstract}
We obtain stochastic duality functions for specific Markov processes using representation theory of Lie algebras. The duality functions come from the kernel of a unitary intertwiner between $*$-representations, which provides (generalized) orthogonality relations for the duality functions. In particular, we consider representations of the Heisenberg algebra and $\su(1,1)$. Both cases lead to orthogonal (self-)duality functions in terms of hypergeometric functions for specific interacting particle processes and interacting diffusion processes.
\end{abstract}

\section{Introduction}
A very useful tool in the study of stochastic Markov processes is duality, where information about a specific process can be obtained from another, dual, process. The concept of duality was introduced in the context of interacting particle systems in \cite{S}, and was later on developed in \cite{L}. For more application of duality see e.g.~\cite{KMP, Sp, C, GKR}.

Two processes are in duality if there exists a duality function, i.e.~a function of both processes such that the expectations with respect to the original process is related to the expectations with respect to the dual process (see Section \ref{sec2} for a precise statement). Recently in \cite{FG} and \cite{RS} orthogonal polynomials of hypergeometric type were obtained as duality functions for several families of stochastic processes, where the orthogonality is with respect to the corresponding stationary measures. These orthogonal polynomials contain the well-known simpler duality functions (in the terminology of \cite{RS}, the classical and cheap duality functions) as limit cases. In \cite{FG}, Franceschini and Giardin\`a use explicit relations between orthogonal polynomials of different degrees, such as raising and lowering formulas, to prove the stochastic duality. In \cite{RS}, Redig and Sau find the orthogonal polynomials using generating functions. With a similar method they also obtain Bessel functions, which are not polynomials, as self-duality function for a continuous process.

The goal of this paper is to demonstrate an alternative method to obtain the orthogonal polynomials (and other `orthogonal' functions) from \cite{FG} and \cite{RS} as duality functions. The method we use is based on representation theory of Lie algebras. This is inspired by \cite{GKRV} and \cite{CGGR}, where representation theory of $\mathfrak{sl}(2,\C)$ and the Heisenberg algebra is used to find (non-orthogonal) duality functions. Roughly speaking, the main idea is to consider a specific element $Y$ in the Lie algebra (or better, enveloping algebra). Realized in two different, but equivalent, representations $\rho$ and $\si$, $\rho(Y)$ and $\si(Y)$ are the generators of two stochastic processes. In case of $\mathfrak{sl}(2,\C)$, $Y$ is closely related to the Casimir operator. The duality functions come from an intertwiner between the two representations. In this paper we consider a similar construction with unitary intertwiners between $*$-representations, so that the duality functions will satisfy (generalized) orthogonality relations.

In Section \ref{sec2} the general method to find duality functions from unitary intertwiners is described. In Section \ref{sec:Heisenberg} the Heisenberg algebra is used to show duality and self-duality for the independent random walker process and a Markovian diffusion process. The self-duality of the diffusion process seems to be new. The (self-)duality functions are Charlier polynomials, Hermite polynomials and exponential functions. In Section \ref{sec:su11} we consider discrete series representation of $\su(1,1)$, and obtain Meixner polynomials, Laguerre polynomials and Bessel functions as (self)-duality functions for the symmetric inclusion process and the Brownian energy process. We would like to point out that these duality functions are essentially the (generalized) matrix elements for a change of base between bases on which elliptic or parabolic Lie group / algebra elements act diagonally, see e.g.~\cite{BK,K}, so in these cases stochastic duality is a consequence of a change of bases in the representation space.

\subsection{Notations and conventions}
By $\N$ we denote the set of nonnegative integers. We use standard notations for shifted factorials and hypergeometric functions as in e.g.~\cite{AAR}. We often write $f(x)$ for a function $x\mapsto f(x)$; the distinction between the function and its values should be clear from the context. For functions $x\mapsto f(x;p)$ depending on one or more parameters $p$, we often omit the parameters in the notation. For a set $E$, we denote by $F(E)$ the vector space of complex-valued functions on $E$. $\mathcal P$ is the vector space consisting of polynomials in one variable. We refer to \cite{KLS} for definitions and properties of the orthogonal polynomials we use in this paper.

\subsection*{Acknowledgements}
I thank Gioia Carinci, Chiara Franceschini, Cristian Giardin\`a and Frank Redig for very helpful discussions and giving valuable comments and suggestions.

\section{Stochastic duality functions from Lie algebra representations} \label{sec2}
In this section we describe the method to obtain stochastic duality functions from $*$-representations of a Lie algebra. This method will be applied in explicit examples in Sections \ref{sec:Heisenberg} and \ref{sec:su11}.

\subsection{Stochastic duality}
Let $X_1=\{\eta_1(t) \mid t>0\}$ and $X_2= \{ \eta_2(t) \mid t>0\}$ be stochastic Markov processes with state spaces $\Om_1$ and $\Om_2$, respectively. These processes are in duality if there exists a duality function $D:\Om_1\times \Om_2 \to \C$ such that for all $t>0$, $\eta_1$ and $\eta_2$, the relation
\[
\mathbb E_{\eta_1}\big[ D(\eta_1(t),\eta_2) \big] = \mathbb E_{\eta_2}\big[ D(\eta_1,\eta_2(t)) \big]
\]
holds, where $\mathbb E_\eta$ represents the expectation. If $X_1=X_2$, the process is called self-dual. Let $L_1$ and $L_2$ be the infinitesimal generators of the two processes. Duality of the processes is equivalent to duality of the generators, i.e.~
\[
[L_1 D(\cdot,\eta_2)](\eta_1) = [L_2 D(\eta_1,\cdot)](\eta_2), \qquad  (\eta_1,\eta_2) \in \Om_1 \times \Om_2.
\]
If $L_1=L_2$, then the operator is self-dual.

In this paper, we consider processes with state space $\Om = E_1 \times \cdots \times E_N$, where each $E_j$ is a subset of $\R$. Furthermore, the generators will be of the form
\begin{equation} \label{eq:L=sum Lij}
L = \sum_{i<j} L_{i,j}
\end{equation}
where $L_{i,j}$ is an operator on $F(E_i \times E_{j})$. This allows us to only consider operators acting on functions in two variables.

\subsection{Lie algebra representations}
Let $\mathfrak g$ be a complex Lie algebra and let $U(\mathfrak g)$ be its universal enveloping algebra. We assume $\mathfrak g$ has a $*$-structure, i.e.~there exists an involution $X \mapsto X^*$ such that
\[
(aX+bY)^* = \overline aX^* + \overline bY^*, \qquad [X,Y]^* =[Y^*,X^*], \qquad X,Y \in \mathfrak g,\ a,b \in \C.
\]
The $*$-structure extends uniquely to a $*$-structure of $U(\mathfrak g)$. Let $\rho$ be a representation of $\mathfrak g$ on the vector space $F(E)$, and assume that $\rho$ is a $*$-representation of $\mathfrak g$ on $\mathcal H=L^2(E,\mu)$, i.e.~$\langle \rho(X)f,g\rangle = \langle f, \rho(X^*) g \rangle$. We assume that $\rho(X)$ is defined on a suitable dense domain $\mathcal D \subset \mathcal H$ for all $X \in \mathfrak g$. $\rho$ extends uniquely to a $*$-representation of $U(\mathfrak g)$ on $\mathcal H$.

If $\rho_1$ and $\rho_2$ are $*$-representations of $\mathfrak g$ on $\mathcal H_1$ and $\mathcal H_2$ respectively, then $\rho$ defined by
\[
\rho(X) = (\rho_1 \tensor \rho_2) (\De(X)), \qquad \De(X) = 1 \tensor X + X \tensor 1 \in U(\mathfrak g)^{\tensor 2}, \quad X \in \mathfrak g,
\]
is a $*$-representation of $\mathfrak g$ on $\mathcal H_1 \tensor \mathcal H_2$ (the Hilbert space completion of the algebraic tensor product of $\mathcal H_1$ and $\mathcal H_2$). Furthermore, $\rho$ can be considered as a representation of $U(\mathfrak g)^{\tensor 2}$ by defining (slightly abusing notation)
\[
\rho(X) = (\rho_1 \tensor \rho_2)(X), \qquad X \in U(\mathfrak g)^{\tensor 2}.
\]
We will often use the notation $\rho=\rho_1\tensor \rho_2$.

Two $*$-representations $\rho_1$ and $\rho_2$ are unitarily equivalent if there exists a unitary operator $\La:\mathcal H_1 \to \mathcal H_2$ such that $\La(\mathcal D_1) \subset \mathcal D_2$ and $\La[\rho_1(X) f] = \rho_2(X) \La(f)$ for all $X \in \mathfrak g$ and $f \in \mathcal D_1$.

\begin{lemma} \label{lem:unitary equivalence}
Let $\rho_j$, $j=1,2$, be representations of $\mathfrak g$ on $F(E_j)$, and $*$-representations of $\mathfrak g$ on $L^2(E_j,\mu_j)$. Suppose \mbox{$K:E_1 \times E_2 \to \C$} is a function with the following properties:
\begin{enumerate}[1.]
\item $[\rho_1(X^*)K(\cdot,y)](x) = [\rho_2(X) K(x,\cdot)](y)$,  for all
$X \in \mathfrak g$ and $(x,y) \in E_1 \times E_2$.
\item The operator $\La:\mathcal D_1 \to L^2(E_2,\mu_2)$ defined by
    \[
    \La f = \left(y \mapsto \int_{E_1} f(x) K(x,y)\, d\mu_1(x) \right),
    \]
    extends to a unitary operator $\La : L^2(E_1,\mu_1) \to L^2(E_2,\mu_2)$.
\end{enumerate}
Then $\rho_1$ and $\rho_2$ are unitarily equivalent $*$-representations with intertwiner $\La$.
\end{lemma}
\begin{proof}
This follows directly from
\[
(\La [\rho_1(X)f])(y)  = \int_{E_1} [\rho_1(X)f](x) K(x,y) \,d\mu_j(x)  = \int_{E_1} f(x) [\rho_1(X^*)K(\cdot,y)](x) \, d\mu(x),
\]
and
\[
[\rho_2(X) (\La f)](y) = \int_{E_1} f(x) [\rho_2(X)K(x,\cdot)](y) \, d\mu_1(x),
\]
using property 1.
\end{proof}

\subsection{Duality from $*$-representations}
For $j\in\{1,\ldots,N\}$ let $\rho_j$ and $\si_j$ be unitarily equivalent $*$-representations of $\mathfrak g$ on $L^2(E_j,\mu_j)$ and $L^2(F_j,\nu_j)$, respectively. We assume that the corresponding unitary intertwiner $\La_j:L^2(E_j,\mu_j) \to L^2(F_j,\nu_j)$ is an integral operator as in Lemma \ref{lem:unitary equivalence}, i.e.
\[
(\La_jf)(y) = \int_{E_j} f(x) K_j(x,y) \,d\mu_j(x) , \qquad \text{for $\nu_j$-almost all } y \in F_j,
\]
for some kernel $K_j:E_j\times F_j \to \C$ satisfying
\begin{equation} \label{eq:rho d=si d}
[\rho_j(X^*)K_j(\cdot,y)](x)= [\si_j(X)K_j(x,\cdot)](y), \qquad (x,y) \in E_j \times F_j, \qquad X \in \mathfrak g.
\end{equation}
Define
\[
\Om_1=E_1\times \cdots \times E_N, \qquad \Om_2= F_1 \times \cdots \times  F_N,
\]
and let $\mu$ and $\nu$ be the product measures on $\Om_1$ and $\Om_2$ given by
\[
\mu=\mu_1 \times \cdots \times \mu_N, \qquad \nu= \nu_1 \times \cdots \times \nu_N,
\]
then $\rho=\rho_1 \tensor \cdots \tensor \rho_N$ and $\si=\si_1 \tensor \cdots \tensor \si_N$ are $*$-representations of $\mathfrak g$ on $L^2(\Om_1,\mu)$ and $L^2(\Om_2,\nu)$.

Suppose that $L_1$ and $L_2$ are self-adjoint operators on $L^2(\Om_1,\mu)$ and $L^2(\Om_2,\nu)$, respectively, given by
\[
L_1=\rho(Y_L), \qquad L_2=\si(Y_L),
\]
for some self-adjoint $Y_L \in U(\mathfrak g)^{\tensor N}$. Then the following result holds.

\begin{theorem} \label{thm:duality}
$L_1$ and $L_2$ are in duality, with duality function given by
\[
D(x,y)=\prod_{j=1}^N K_j(x_j,y_j), \qquad x=(x_1,\ldots,x_N) \in \Om_1, \ y=(y_1,\ldots,y_N) \in \Om_2.
\]
\end{theorem}
\begin{proof}
Write $Y_L = \sum Y_{(1)} \tensor \cdots \tensor Y_{(N)}$, with $Y_{(j)} \in U(\mathfrak g)$. It is enough to verify that
\[
[\rho_j(Y^*_{(j)}) K_j(\cdot,y_j)](x_j) = [\si_j(Y_{(j)} K_j(x_j,\cdot)](y_j), \qquad (x_j,y_j) \in E_j \times F_j,
\]
for $j=1,2$.
Since we have $Y_{(j)}= Y_{j,1}Y_{j,2} \cdots Y_{j,k_j}$ for certain $Y_{j,i} \in \mathfrak g$, the result follows from \eqref{eq:rho d=si d}.
\end{proof}

In the following sections we apply Theorem \ref{thm:duality} using explicit representations in terms of difference operators or differential operators. The main problem is finding explicitly the appropriate intertwiner. The algebra element $Y_L$ will always have a specific form corresponding to \eqref{eq:L=sum Lij};
\[
Y_L = \sum_{i<j}  Y_{i,j} \qquad \text{with } Y \in U(\mathfrak g)^{\tensor 2}.
\]
Here we use the following notation: for $Y=\sum Y_{(1)} \tensor Y_{(2)} \in U(\mathfrak g)^{\tensor 2}$ we denote by $Y_{i,j} \in U(\mathfrak g)^{\tensor N}$ the element
\[
Y_{i,j} = \sum 1 \tensor \cdots \tensor 1 \tensor Y_{(1)} \tensor 1 \tensor \cdots \tensor 1 \tensor Y_{(2)} \tensor 1 \cdots \tensor 1,
\]
with $Y_{(1)}$ in the $i$-th factor and $Y_{(2)}$ in the $j$-th factor.

\section{The Heisenberg algebra} \label{sec:Heisenberg}
To illustrate how the method from the previous section is applied, we use the Heisenberg Lie algebra to obtain duality functions for two stochastic processes. Let us first describe the processes.

The independent random walker process IRW is a Markov jump process where particles move independently between $N$ sites, and each site can contain an arbitrary number of particles. Particles jump from site $i$ to site $j$ with rate proportional to the number of particles $n_i$ at site $i$. The generator of this process is the difference operator acting on appropriate function in $F(\N^N)$ given by
\begin{equation} \label{eq:L IRW}
L^{\mathrm{IRW}}f(n) = \sum_{1 \leq i < j \leq N} n_i \left( f(n^{i,j}) - f(n) \right) +  n_j \left( f(n^{j,i} - f(n) \right), \qquad n \in \N^N,
\end{equation}
Here $n^{i,j} = n+e_i-e_j$, where $e_i$ the standard basis vector with $1$ as $i$-th component and all other component are $0$.

The second process is a Feller diffusion process on $\R^N$ with a constant diffusion matrix, and a drift vector which is a function of the difference of pairs of coordinates. It can be considered as $N$ Brownian motions which are attracted to each other with a rate proportional to their distances.
The generator is a differential operator on appropriate functions in $F(\R^N)$ given by
\begin{equation} \label{eq:L DIF}
L^{\mathrm{DIF}}f(x) = c\sum_{1 \leq i < j \leq N} \left( \frac{\partial}{\partial x_i}-  \frac{ \partial}{\partial x_j} \right)^2 f(x)- (x_i-x_j)\left(  \frac{\partial}{\partial x_i}- \frac{ \partial}{\partial x_j} \right) f(x),
\end{equation}
where $x \in \R^N$ and $c>0$.

Note that both generators have the form \eqref{eq:L=sum Lij}.\\

The Heisenberg algebra $\mathfrak h$ is the Lie algebra with generators $a,a^\dagger,Z$ satisfying
\begin{equation} \label{eq:commutations heisenberg}
[a,Z]=[a^\dagger,Z]=0, \qquad [a^\dagger,a]=Z.
\end{equation}
The $*$-structure is given by $a^*=a^\dagger$, $(a^\dagger)^*=a$ and $Z^*=Z$. Note that $Z$ is a central element.
$\mathfrak h$ has a representation $\rho_{c}$ with parameter $c>0$ on $F(\N)$ given by
\begin{equation} \label{eq:representation h}
\begin{split}
[\rho_{c}(a)f](n) &= n f(n-1), \\
[\rho_{c}(a^\dagger)f](n) &=  cf(n+1),\\
[\rho_{c}(Z)f)](n) &= cf(n),
\end{split}
\end{equation}
where $f(-1)=0$. Then $\rho_{c}$ is a $*$-representation on the weighted $L^2$-space $\mathcal H_c = \ell^2(\N,w_c)$ consisting of functions in $F(\N)$ that have finite norm with respect to the inner product
\[
\langle f,g\rangle = \sum_{n \in \N} w_c(n)\, f(n) \overline{g(n)}, \qquad w_c(n) = \frac{c^n}{n!}e^{-c}.
\]
$\rho_c$ is an unbounded representation, with dense domain the set $F_0(\N)$ consisting of finitely supported functions.

Define $Y \in U(\mathfrak h)^{\tensor 2}$ by
\begin{equation} \label{eq:Y in h}
Y= (1 \tensor a - a \tensor 1 )(a^\dagger \tensor 1 - 1 \tensor a^\dagger).
\end{equation}
This element gives us the relation with the system of independent random walkers.
\begin{lemma} \label{lem:L IRW=sum Y}
For $c>0$, let $\rho$ be the tensor product representation $\rho = \rho_{c} \tensor \cdots \tensor \rho_{c}$ of $\mathfrak h$ on $\mathcal H_c^{\tensor N}$, then
\[
L^{\mathrm{IRW}} = c^{-1}\sum_{1 \leq i < j \leq N} \rho(Y_{i,j}).
\]
\end{lemma}
\begin{proof}
It suffices to consider $(\rho_{c} \tensor \rho_{c}) (Y)$ acting on functions in two variables $n_1$ and $n_2$. From \eqref{eq:representation h} and \eqref{eq:Y in h} we find
\[
\begin{split}
[\rho_{c}\tensor \rho_{c}&(Y)f](n_1,n_2) = \\ &c n_1 \Big( f(n_1-1,n_2+1) - f(n_1,n_2) \Big) + c n_2 \Big( f(n_1+1,n_2-1) - f(n_1,n_2) \Big).
\end{split}
\]
This corresponds to the term $(i,j)=(1,2)$ in \eqref{eq:L IRW}.
\end{proof}
\begin{remark}
We can also consider the tensor product representation $\rho_{c_1} \tensor \cdots \tensor \rho_{c_N}$ with (possibly) $c_i \neq c_j$. This leads to a generator of a Markov process depending on $N$ different parameters. However, to prove self-duality it seems crucial to assume $c_i=c$ for all $i$, see Lemma \ref{lem:Y->Y+R} later on.
\end{remark}

\subsection{Charlier polynomials and self-duality of IRW}
The Charlier polynomials are defined by
\[
C_n(x;c) = \rFs{2}{0}{-n,-x}{\mhyphen}{-\frac{1}{c}}.
\]
They form an orthogonal basis for $\ell^2(\N, w_a)$, with orthogonality relations
\[
\sum_{x\in \N}w_c(x)\, C_m(x;c) C_n(x;c) = \de_{mn} c^{-n}n!, \qquad a>0,
\]
and they have the following raising and lowering property,
\begin{equation} \label{eq:relations Charlier}
\begin{split}
nC_{n-1}(x;c) &= cC_n(x;c) - c C_{n}(x+1;c), \\
cC_{n+1}(x;c) &=c C_n(x;c)- x C_{n}(x-1;c).
\end{split}
\end{equation}
They are self-dual, i.e.~$C_n(x;c)= C_x(n;c)$.

Let us consider the actions of $a$ and $a^\dagger$ on the Charlier polynomials
\[
C(n,x;c)= e^{c}C_n(x;c).
\]
The reason for this normalization is to obtain a unitary intertwiner with $C(n,x)$ as a kernel later on. Using the raising and lowering properties \eqref{eq:relations Charlier} we obtain
\begin{equation} \label{eq:actions on Cnx}
\begin{split}
[\rho_c(a) C(\cdot,x)](n) &= n C(n-1,x) = c C(n,x) - c C(n,x+1),\\
[\rho_c(a^\dagger) C(\cdot,x)](n) & = c C(n+1,x) = c C(n,x) - x C(n,x-1).
\end{split}
\end{equation}
Note that the actions on the $x$-variable are similar to the actions of $Z-a$ and $Z-a^\dagger$ in the $n$-variable.
This motivates the definition of the following isomorphism.
\begin{lemma} \label{lem:te h-> h}
The assignments
\[
\te(a)=Z-a,\quad \te(a^\dagger) = Z-a^\dagger, \quad  \te(Z) = Z,
\]
extend uniquely to a Lie algebra isomorphism $\te:\mathfrak h \to \mathfrak h$.
\end{lemma}
\begin{proof}
The proof consists of checking the commutation relations \eqref{eq:commutations heisenberg}, which is a straightforward computation.
\end{proof}
Note that
\[
[\rho_c(\te(a)) C(\cdot;x)](n) = c\,C(n,x+1), \qquad [\rho_c(\te(a^\dagger)) C(\cdot;x)](n) = x\, C(n,x-1).
\]
Furthermore, $\te$ preserves the $*$-structure, i.e.~$\te(X^*) = \te(X)^*$. Clearly, $\rho_c\circ \te$ is again a $*$-representation of $\mathfrak h$ on $\mathcal H_c$.  We will use the Charlier polynomials to construct a unitary intertwiner between $\rho_c$ and $\rho_c \circ \te$.
\begin{prop} \label{prop:intertwiner Charlier}
Define the operator $\La:F_0(\N)\to F(\N)$ by
\[
(\La f)(x) = \sum_{n \in \N} w_c(n) f(n) C(n,x;c),
\]
then $\La$ extends to a unitary operator $\La:\mathcal H_c \to \mathcal H_c$, and intertwines $\rho_c$ with $\rho_c\circ \te$. Furthermore, the kernel $C(n,x)$ satisfies
\begin{equation} \label{eq:rhoC=siC}
[\rho_c(X^*) C(\cdot,x)](n) = [\rho_c(\te(X)) C(n,\cdot)](x), \qquad X \in \mathfrak h.
\end{equation}
\end{prop}
\begin{proof}
The cheap duality functions $\de_m(n) = \frac{ \de_{m,n} }{w_c(n)}$ form an orthogonal basis for $\mathcal H_c$ with squared norm $\|\de_m\|^2 = \frac{1}{w_c(m)}$. Applying $\La$ to $\de_m$ gives
\[
\La(\de_m)(x) = C(m,x).
\]
From the orthogonality relations for the Charlier polynomials we find that the squared norm of $C(m,x)$ is $\|C(m,\cdot)\|^2 = \frac{1}{w_c(m)}$. So $\La$ maps an orthogonal basis to another orthogonal basis with the same norm, hence $\La$ is unitary.

To apply Lemma \ref{lem:unitary equivalence} we need to verify that \eqref{eq:rhoC=siC} is satisfied. It is enough to do this for $X = a, a^\dagger, Z$. Using \eqref{eq:actions on Cnx} and $a^*=a^\dagger$ we see that
\begin{gather*}
[\rho_c(a^*) C(\cdot,x)](n)=  c C(n,x) - x C(n,x-1) = [\rho_c(Z-a)C(n,\cdot)](x), \\
[\rho_c((a^\dagger)^*) C(\cdot,x)](n)  =  c C(n,x) - c C(n,x+1) = [\rho_c(Z-a^\dagger)C(n,\cdot)](x).
\end{gather*}
The action of $Z$ is clear. Now the result follows from the definition of $\te$, see Lemma \ref{lem:te h-> h}.
\end{proof}
We are almost ready to prove self-duality for IRW, but first we need to know the image of $Y$, see \eqref{eq:Y in h}, under the isomorphism $\te \tensor \te$.
\begin{lemma} \label{lem:Y->Y+R}
The following identity in $\mathfrak h \tensor \mathfrak h$ holds:
\[
\theta \tensor \theta (Y) = Y + R
\]
with
\[
\begin{split}
R =\, &  1 \tensor Za^\dagger - Z \tensor a^\dagger + Za^\dagger\tensor 1 - a^\dagger \tensor Z + 1 \tensor aZ - Z \tensor a + aZ \tensor 1 - a \tensor Z\\
& + 2 \, Z \tensor Z - Z^2 \tensor 1 - 1 \tensor Z^2.
\end{split}
\]
In the representation $\rho_c \tensor \rho_c$, the element $R$ is the zero operator on $\mathcal H_c \tensor \mathcal H_c$.
\end{lemma}
\begin{proof}
After a somewhat tedious computation using the definition of $\te$ in Lemma \ref{lem:te h-> h}, we find the explicit expression for $\te \tensor \te (Y)$. Using $\rho_c(Z) = c\, \mathrm{Id}$ it follows that $\rho_c \tensor \rho_c (R) = 0$.
\end{proof}

We can now apply Theorem \ref{thm:duality} with $\rho=\rho_{c} \tensor\cdots \tensor \rho_{c}$ and $\si=\rho \circ (\te \tensor \cdots \tensor \te)$. Using Lemma \ref{lem:Y->Y+R} we find $\si(Y_{i,j})=\rho(Y_{i,j})$, and then it follows that $\sum \si(Y_{i,j}) = L^{\mathrm{IRW}}$, see Lemma \ref{lem:L IRW=sum Y}. So we obtain the well-known self-duality of the independent random walker process. Here the duality function is a product of Charlier polynomials.
\begin{theorem}
The operator $L^{\mathrm{IRW}}$ given by \eqref{eq:L IRW} is self-dual, with duality function
\[
\prod_{j=1}^N C(n_j,x_j;c), \qquad c>0.
\]
\end{theorem}

\subsection{Hermite polynomials and duality between IRW and the diffusion process}
The Hermite polynomials are defined by
\[
H_n(x) = (2x)^n \rFs{2}{0}{-\frac{n}{2}, -\frac{n-1}{2}}{\mhyphen}{-\frac{1}{x^2}}.
\]
They form an orthogonal basis for $L^2(\R, e^{-x^2} dx)$, with orthogonality relations
\[
\frac{1}{\sqrt{\pi}}\int_\R H_m(x) H_n(x) e^{-x^2}\, dx = \de_{mn} 2^n n!,
\]
and they have the following lowering and raising properties
\begin{equation} \label{eq:diff Hn}
\begin{gathered}
\frac{d}{dx}H_n(x) = 2n H_{n-1}(x), \\ \left(-\frac{d}{dx}+2x\right)H_n(x) = H_{n+1}(x).
\end{gathered}
\end{equation}

With the lowering and raising operators for the Hermite polynomials we can realize $a$ and $a^\dagger$ as differential operators.
We define
\[
H(n,x;c) = e^{\frac{c}{2}} (2c)^{-\frac{n}{2}} H_n\left( \tfrac{x}{\sqrt{2c}} \right).
\]
Using the representation $\rho$ \eqref{eq:representation h} and the differential operators \eqref{eq:diff Hn} we find the following result.
\begin{lemma}  \label{eq:rhoc on H}
The Hermite polynomials $H(n,x)$ satisfy
\begin{gather*}
[\rho_c(a)H(\cdot,x)](n) = c\frac{\partial}{\partial x} H(n,x),\\
[\rho_c(a^\dagger) H(\cdot,x)](n) = \left(x-c \frac{\partial}{\partial x} \right)H(n,x).
\end{gather*}
\end{lemma}

Next we define an unbounded $*$-representation $\si_c$ of $\mathfrak h$ on the Hilbert space $\mathrm H_c= L^2(\R, w(x;c)dx)$, where
\[
w(x;c) = \frac{e^{-\frac{x^2}{2c}}}{\sqrt{2c\pi}}.
\]
The Hermite polynomials $H(n,x)$ form an orthogonal basis for $\mathrm H_c$, with squared norm $\|H(n,\cdot)\|^2 = \frac1{w_c(n)}$. We define the representation $\si_c$ by
\[
\begin{split}
[\si_c(a) f](x) &= xf(x) - c \frac{\partial }{\partial x}f(x),\\
[\si_c(a^\dagger) f](x) &= c \frac{ \partial }{\partial x}f(x),\\
[\si_c(Z)f](x) &= cf(x).
\end{split}
\]
As a dense domain we take the set of polynomials $\mathcal P$.
\begin{prop}
Define $\La:F_0(\N)\to F(\R)$ by
\[
(\La f)(x) = \sum_{n \in \N} w_c(n) f(n) H(n,x;c),
\]
then $\La$ extends to a unitary operator $\La:\mathcal H_c \to \mathrm H_c$ intertwining $\rho_c$ with $\si_c$. Furthermore, the kernel $H(n,x)$ satisfies
\[
[\rho_c(X^*)H(\cdot,x)](n) = [\si_c(X) H(n,\cdot)](x), \qquad X \in \mathfrak h.
\]
\end{prop}
\begin{proof}
Unitarity of $\La$ is proved in the same way as in Proposition \ref{prop:intertwiner Charlier}. The intertwining property for the kernel follows from Lemma \ref{eq:rhoc on H}. Lemma \ref{lem:unitary equivalence} then shows that $\La$ intertwines $\rho_c$ and $\si_c$.
\end{proof}

Similar as in Lemma \ref{lem:L IRW=sum Y} we find that the generator $L^{\mathrm{DIF}}$ defined by \eqref{eq:L DIF} is the realization of $Y$ define by \eqref{eq:Y in h} on the Hilbert space $\mathrm H_{c}^{\tensor N}$.
\begin{lemma} \label{lem:L DIF=sum Y}
For $c >0$ define $\si = \si_{c}\tensor \cdots \tensor \si_{c}$, then
\[
L^{\mathrm{DIF}} = c^{-1}\sum_{1 \leq i<j\leq N} \si_c(Y_{i,j}).
\]
\end{lemma}

Finally, applying Theorem \ref{thm:duality} we obtain duality between $L^{\mathrm{IRW}}$ and $L^{\mathrm{DIF}}$, with duality function given by Hermite polynomials.
\begin{theorem}
$L^{\mathrm{IRW}}$ and $L^{\mathrm{DIF}}$ are in duality, with duality function given by
\[
\prod_{j=1}^N H(n_j,x_j;c).
\]
\end{theorem}
\begin{remark}
This duality between $L^{\mathrm{IRW}}$ and $L^{\mathrm{DIF}}$ was also obtained in \cite[Remark 3.1]{GKRV}, but Hermite polynomials are not mentioned there. Hermite polynomials of even degree have appeared as duality functions in \cite[\S4.1.1]{FG}; this can be considered as a special case of duality involving Laguerre polynomials, see \cite[\S4.2.1]{FG} or Theorem \ref{thm:SIP - BEP}.
\end{remark}

\subsection{The exponential function and self-duality of the diffusion process}

To show self-duality of the differential operator $L^{\mathrm{DIF}}$, the following isomorphism is useful.
\begin{lemma}
The assignments
\[
\te(a) = \frac{1}{2}(a-a^\dagger), \qquad \te(a^\dagger) = i(a+a^\dagger), \qquad \te(Z) = iZ,
\]
extend uniquely to a Lie algebra isomorphism $\te:\mathfrak h \to \mathfrak h$.
\end{lemma}
\begin{proof}
We just need to check commutation relations, which is a direct calculation.
\end{proof}
Observe that in the representation $\si_c$, $\te(a)$ and $\te(a^\dagger)$ are the operators
\[
\si_c(\te(a)) = \frac{x}{2}-c\frac{\partial}{\partial x}, \qquad \si_c(\te(a^\dagger)) = ix.
\]
The kernel of the (yet to be defined) intertwining operator is the exponential function
\[
\phi(x,y;c) = \exp\left(\frac{x^2+y^2}{4c}-\frac{ixy}{c}\right), \qquad x,y \in \R.
\]

\begin{lemma} \label{lem:sic on phi}
The function $\phi(x,y)$ satisfies
\[
\begin{split}
[\si_c(\te(a)) \phi(\cdot,y)](x) & =  iy\, \phi(x,y),\\
[\si_c(\te(a^\dagger)) \phi(\cdot,y)](x) & = \left(\frac{y}{2}-c\frac{\partial}{\partial y}\right)\phi(x,y),
\end{split}
\]
\end{lemma}
\begin{proof}
From
\[
c \frac{\partial}{\partial x}\phi(x,y) = \left(\frac{x}{2}-iy\right)\phi(x,y)
\]
we obtain
\[
[\si_c(\te(a)) \phi(\cdot,y)](x) = \left(\frac{x}{2}-c\frac{\partial}{\partial x}\right) \phi(x,y) = iy\, \phi(x,y).
\]
Using symmetry in $x$ and $y$, we find
\[
[\si_c(\te(a^\dagger)) \phi(\cdot,y)](x) = ix\,\phi(x,y) =\left(\frac{y}{2}-c\frac{\partial }{\partial y}\right)\phi(x,y). \qedhere
\]
\end{proof}
Now we can show that the integral operator with $\phi$ as a kernel is the desired intertwining operator.
\begin{prop}
Define $\La:\mathcal P\to F(\R)$ by
\[
(\La f)(y) =  \int_\R f(x) \phi(x,y;c) w(x;c)\, dx,
\]
then $\La$ extends to a unitary operator $\La: \mathrm H_c \to \mathrm H_c$ intertwining $\si_c$ with $\si_c$. Furthermore, the kernel $\phi(x,y)$ satisfies
\[
[\si_c(X^*)\phi(\cdot,y)](x) = [\si_c(X) \phi(x,\cdot)](y), \qquad X \in \mathfrak h.
\]
\end{prop}
\begin{proof}
Unitarity of $\La$ follows from unitarity of the Fourier transform. From Lemma \ref{lem:sic on phi} we find
\[
[\si_c(\te(X^*))\phi(\cdot,y)](x) = [\si_c(\te(X)) \phi(x,\cdot)](y), \qquad X \in \mathfrak h,
\]
which proves the result, since $\te$ is an isomorphism.
\end{proof}

By applying Theorem \ref{thm:duality} and using Lemma \ref{lem:L DIF=sum Y} we obtain self-duality of $L^{\mathrm{DIF}}$ \eqref{eq:L DIF}.
\begin{theorem}
$L^{\mathrm{DIF}}$ is self-dual, with duality function given by
\[
\prod_{j=1}^N \phi(x_j,y_j;c).
\]
\end{theorem}

\section{The Lie algebra $\su(1,1)$} \label{sec:su11}

In this section we use the Lie algebra $\su(1,1)$ to find duality functions for two stochastic processes.
The first one is the symmetric inclusion process SIP($k$), $k \in \R_{>0}^N$, which is a Markov jump process on $N$ sites, where each site can contain an arbitrary number of particles. Jumps between two sites, say $i$ and $j$, occur at a rate proportional to the number of particles $n_i$ and $n_j$. The generator of this process is given by
\begin{equation} \label{eq:L SIP}
L^{\mathrm{SIP}} f(n) = \sum_{1 \leq i < j \leq N} n_i(2k_j+n_j) \left(f(n^{i,j}) - f(n) \right) + n_j(2k_i+n_i) \left(f(n^{j,i}) - f(n) \right),
\end{equation}
with $n=(n_1,\ldots,n_N) \in \N^N$.

The second process is the Brownian energy process BEP($k$), $k \in \R_{>0}^N$, which is a Markov diffusion process that describes the evolution of a system of $N$ particles that exchange energies. The energy of particle $i$ is $x_i>0$. The generator is given by
\begin{equation} \label{eq:L BEP}
L^{\mathrm{BEP}} f(x) = \sum_{1 \leq i<j \leq N} x_i x_j \left( \frac{\partial}{\partial x_i}  - \frac{\partial}{\partial x_j} \right)^2f(x) - 2(k_i x_i - k_j x_j) \left( \frac{\partial}{\partial x_i}  - \frac{\partial}{\partial x_j}  \right) f(x),
\end{equation}
with $x=(x_1,\ldots,x_N) \in \R_{>0}^N$.\\

The Lie algebra $\mathfrak{sl}(2,\C)$ is generated by $H,E,F$ with commutation relations
\[
[H,E]=2E, \quad [H,F]=-2F, \quad [E,F]=H.
\]
The Lie algebra $\su(1,1)$ is $\mathfrak{sl}(2,\C)$ equipped with the $\ast$-structure
\[
H^*=H, \quad E^*=-F, \quad F^*=-E.
\]
The Casimir element $\Om$ is a central self-adjoint element of the universal
enveloping algebra $U\big(\su(1,1)\big)$ given by
\begin{equation} \label{Casimir}
 \Omega = \frac12H^2 + EF + FE.
\end{equation}
Note that $\Om^*=\Om$.\\

We consider the following representation of $\su(1,1)$. Let $k>0$ and $0<c<1$. The
representation space is the weighted $L^2$-space $\mathcal H_{k,c}=\ell^2(\N,w_{k,c})$ consisting of functions on $\N$ which have finite norm with respect to the inner product
\[
\langle f,g\rangle = \sum_{n \in \N} w_{k,c}(n)\, f(n) \overline{g(n)}, \qquad w_{k,c}(n) =  \frac{ (2k)_n }{n!}c^n (1-c)^{2k}.
\]
The actions of the generators are given by
\begin{equation} \label{eq:actions HEF}
\begin{split}
[\pi_{k,c}(H) f](n) =&\ 2(k+n) f(n),  \\ [\pi_{k,c}(E)f] (n) =&\
\frac{n}{\sqrt{c}} f(n-1),  \\ [\pi_{k,c}(F)f] (n) =&\ -\sqrt{c}\,(2k+n)
f(n+1),
\end{split}
\end{equation}
where $f(-1)=0$.
This defines an unbounded $*$-representation on $\mathcal H_{k,c}$, with dense domain $F_0(\N)$. Note that $[\pi_{k,c}(\Om) f](n) = 2k(k-1) f(n)$.

\begin{remark}
For $0<c_1,c_2<1$ define a unitary operator $I:\mathcal H_{k,c_1}\to \mathcal H_{k,c_2}$ by
\[
(If)(n) = \left( \frac{c_1}{c_2}\right)^{n/2} f(n).
\]
Then $I\circ \pi_{k,c_1} = \pi_{k,c_2} \circ I$, so for fixed $k>0$ all representations $\pi_{k,c}$, $0<c<1$, are unitarily equivalent (we can even take $c \geq 1$ if we omit the factor $(1-c)^{2k}$ from the weight function $w_{k,c}$). From here on we assume that $c$ is a fixed parameter, and just write $\pi_k$ and $\mathcal H_k$ instead of $\pi_{k,c}$ and $\mathcal H_{k,c}$.
\end{remark}

The generator $L^{\mathrm{SIP}}$ is related to the Casimir $\Om$. Recall that the coproduct $\De$ is given by $\De(X)  = 1\tensor X + X \tensor 1$, and $\De$ extends as an algebra morphism to $U(\su(1,1))$. This gives
\[
\De(\Om) = 1\tensor \Om + \Om \tensor 1 + H \tensor H + 2F \tensor E + 2E \tensor F.
\]
We set
\begin{equation} \label{eq:Y}
Y = \frac12\Big(1\tensor \Om + \Om \tensor 1 - \De(\Om)\Big).
\end{equation}
The relation to the symmetric inclusion process is as follows.
\begin{lemma} \label{lem:LSIP su11}
For $k=(k_1,\ldots,k_N) \in \R_{>0}^N$ define $\pi_k = \pi_{k_1}\tensor \cdots \tensor \pi_{k_N}$, then
\[
L^{\mathrm{SIP}} =\sum_{1 \leq i<j \leq N} \pi_k (Y_{i,j}) + 2k_ik_j
\]
\end{lemma}
\begin{proof}
Consider $L_{1,2}= \pi_{k_1}\tensor \pi_{k_2}(Y)+2k_1k_2$. It suffices to show that $L_{1,2}$ gives the term $(i,j) = (1,2)$ in \eqref{eq:L SIP}. From \eqref{eq:Y} and \eqref{eq:actions HEF} we find that $L_{1,2}$ acts on $f \in \mathcal H_{k_1} \tensor \mathcal H_{k_2}$ by
\[
\begin{split}
[L_{1,2} f](n_1,n_2) &= n_1(2k_2+n_2) [f(n_1-1,n_2+1) - f(n_1,n_2)] \\
& \quad + n_2(2k_1+n_1) [f(n_1+1,n_2-1) - f(n_1,n_2)],
\end{split}
\]
which is the required expression.
\end{proof}

In order to obtain duality functions we consider eigenfunctions of the following self-adjoint element as in \cite{KJ1}:
\begin{equation} \label{eq:Xa}
X_a = -a H + E - F \in \su(1,1), \qquad a \in \R.
\end{equation}
Depending on the value of $a$ this is either an elliptic element ($|a|>1$), parabolic element ($|a|=1$), or hyperbolic element ($|a|<1$), corresponding to the associated one-parameter subgroups in $\mathrm{SU}(1,1)$.

\subsection{Meixner polynomials and self-duality for SIP} \label{ssec:Meixner}
The Meixner polynomials are
defined by
\begin{equation} \label{Meixner}
M_n(x;\beta,c) = \rFs{2}{1}{-n,-x}{\beta}{1-\frac{1}{c}}.
\end{equation}
These are self-dual: $M_n(x;\be,c) = M_x(n;\be,c)$ for $x \in \N$.
For $\beta>0$ and $0<c<1$, the Meixner polynomials are orthogonal with respect to a positive measure on $\N$,
\[
\sum_{x=0}^\infty \frac{(\be)_x c^x} { x! } M_m(x) M_n(x) = \de_{mn} \frac{ c^{-n} n!}{(\be)_n (1-c)^\be},
\]
and the polynomials form a basis for the corresponding Hilbert space.
The three-term recurrence relation for the Meixner polynomials is
\[
(c-1)(x+\tfrac12\be) M_n(x) = c(n+\be) M_{n+1}(x) - (c+1)(n+\tfrac12\be) M_n(x) + n M_{n-1}(x).
\]
Using the self-duality this also gives a difference equation in the $x$-variable for the Meixner polynomials. \\

We set $a(c) = \frac{1+c}{2\sqrt{c}}$, so that $a(c)>1$. The action of
\[
X_{a(c)}= - \frac{1+c}{2\sqrt{c}}H+E-F
\]
on $f \in \mathcal H_{k}$ is given by
\[
[\pi_k(X_{a(c)}) f](n) = \sqrt{c}(2k+n) f(n+1) - \frac{1+c}{\sqrt{c}} (k+n) f(n) + \frac{n}{\sqrt{c}} f(n-1).
\]
\begin{lemma}
The Meixner polynomials $M(n,x;k,c) =  M_n(x;2k,c)$ are eigenfunctions of $\pi_k(X_{a(c)})$, i.e.
\[
[\pi_k(X_{a(c)}) M(\cdot,x)](n) = \frac{c-1}{\sqrt{c}}(x+k)\, M(n,x), \qquad x \in \N.
\]
\end{lemma}
\begin{proof}
This follows from the three-term recurrence relation for the Meixner polynomials.
\end{proof}

Using the difference equation for the Meixner polynomials, we can realize $H$ as a difference operator acting on $M(n,x)$ in the $x$-variable.
\begin{lemma} \label{lem:action of H}
The following identity holds:
\[
\left[ \pi_k(H) M(\cdot,x)\right](n) = -\frac{2c}{1-c} (2k+x) M(n,x+1)  + \frac{1+c}{1-c} 2(k+x)  M(n,x) -\frac{2x}{1-c}   M(n,x-1).
\]
\end{lemma}
\begin{proof}
We use the difference equation for $M(n,x)$, which is, by self-duality, equivalent to the three-term recurrence relation:
\[
(c-1) (n+k) \, M(n,x) = c (2k+x) M(n,x+1)  - (1+c)(k+x)  M(n,x) + x M(n,x-1).
\]
Since $[\pi_k(H)M(\cdot,x)](n) = 2(k+n) M(n,x)$ the result follows.
\end{proof}

With the actions of $X_{a(c)}$ and $H$ on Meixner polynomials, it is possible to express $E$ and $F$ acting on $M(n,x)$ as three-term difference operators in the variable $x$. This leads to a representation by difference operators in $x$, in which the basis elements $H$, $E$ and $F$ all act by three-term difference operators. Having actions of $H$, $E$ and $F$, we can express, after a large computation, $\De(\Om)$ in terms of difference operators in two variables $x_1$ and $x_2$. We prefer, however, to work with a simpler representation in which $H$ acts as a multiplication operator, and $E$ and $F$ as one-term difference operators.
Note that the action of $X_{a(c)}$ in the $x$-variable corresponds up to a constant to the action of $H$ in the $n$-variable, i.e.~it is a multiplication operator. We can make a new $\mathfrak{sl}(2,\C)$-triple with $X_{a(c)}$ playing the role of $H$, see \cite[\S3.2]{GK} The following results give the corresponding isomorphism.
\begin{lemma}\label{lem:isomorphism phi}
Define elements $H_{\sqrt{c}}, E_{\sqrt{c}}, F_{\sqrt{c}} \in \su(1,1)$ by
\[
\begin{split}
H_{\sqrt{c}} &= \frac{1+c}{1-c}H  - \frac{2\sqrt{c}}{1-c}E  + \frac{2\sqrt{c}}{1-c} F, \\
E_{\sqrt{c}} & = -\frac{\sqrt{c}}{1-c} H + \frac{1}{1-c} E - \frac{c}{1-c}F,\\
F_{\sqrt{c}} & = \frac{\sqrt{c}}{1-c} H - \frac{c}{1-c} E + \frac{1}{1-c} F,
\end{split}
\]
then the assignments
\[
\te_{\sqrt{c}}(H) = H_{\sqrt{c}}, \quad \te_{\sqrt{c}}(E)=E_{\sqrt{c}}, \quad \te_{\sqrt{c}}(F) = F_{\sqrt{c}},
\]
extend to a Lie algebra isomorphism $\te_{\sqrt{c}}: \su(1,1) \to \su(1,1)$ with inverse $(\theta_{\sqrt{c}})^{-1} = \theta_{-\sqrt{c}}$. Furthermore, $\theta_{\sqrt{c}}(\Om) = \Om$.
\end{lemma}
\begin{proof}
We need to check the commutation relations, which is a straightforward computation.
\end{proof}
Note that $\te_{\sqrt{c}}$ preserves the $*$-structure, i.e.~$\te_{\sqrt{c}}(X^*) = \te_{\sqrt{c}}(X)^*$. \\

Observe that $H_{\sqrt{c}} = \frac{2\sqrt{c}}{c-1}X_{a(c)}$, so we defined $H_{\sqrt{c}}$ in such a way that
\[
[\pi_k(H_{\sqrt{c}}) M(\cdot,x)](n) = 2(k+x) M(n,x).
\]
By Lemma \ref{lem:isomorphism phi},
\[
H = \theta_{\sqrt{c}}^{-1}(H_{\sqrt{c}}) = \theta_{-\sqrt{c}}(H_{\sqrt{c}}) = \frac{1+c}{1-c}H_{\sqrt{c}}  + \frac{2\sqrt{c}}{1-c}E_{\sqrt{c}}  - \frac{2\sqrt{c}}{1-c} F_{\sqrt{c}}.
\]
This gives
\begin{equation} \label{eq:EFc}
\begin{gathered}
E_{\sqrt{c}}-F_{\sqrt{c}} = \frac{1-c}{2\sqrt{c}} H - \frac{ 1+c}{2\sqrt{c}}H_{\sqrt{c}},\\
E_{\sqrt{c}}+F_{\sqrt{c}}  = \frac{1-c}{4\sqrt{c}} [H,H_{\sqrt{c}}],
\end{gathered}
\end{equation}
which shows that we can express $E_{\sqrt{c}}$ and $F_{\sqrt{c}}$ completely in terms of $H_{\sqrt{c}}$ and $H$. This allows us to write down explicit actions of $E_{\sqrt{c}}$ and $F_{\sqrt{c}}$ acting on $M(n,x)$ in the $x$-variable. This then shows that $M(n,x)$ has the desired intertwining properties.
\begin{lemma} \label{eq:actions on M(n,x)}
The functions $M(n,x)$ satisfy
\[
\begin{split}
[\pi_k(H_{\sqrt{c}})M(\cdot,x)](n) & = 2(k+x) M(n,x),\\
[\pi_k(E_{\sqrt{c}})M(\cdot,x)](n) & = -\frac{x}{\sqrt{c}} M(n,x-1),\\
[\pi_k(F_{\sqrt{c}})M(\cdot,x)](n) & = \sqrt{c}(2k+x) M(n,x+1).
\end{split}
\]
\end{lemma}
\begin{proof}
We already know the action of $H_{\sqrt{c}}$. Using \eqref{eq:EFc} and Lemma \ref{lem:action of H} for the action of $H$ we find
\begin{gather*}
\left[\pi_k(E_{\sqrt{c}}-F_{\sqrt{c}}) M(\cdot,x)\right](n) = -\sqrt{c}(2k+x) M(n,x+1)  - \frac{x}{\sqrt{c}} M(n,x-1),\\
\left[\pi_k(E_{\sqrt{c}}+F_{\sqrt{c}}) M(\cdot,x)\right](n) =  \sqrt{c}(2k+x) M(n,x+1) - \frac{x}{\sqrt{c}} M(n,x-1),
\end{gather*}
which gives the actions of $E_{\sqrt{c}}$ and $F_{\sqrt{c}}$.
\end{proof}
Now we are ready to define the intertwiner.
\begin{prop} \label{prop:intertwiner}
The operator $\La:F_0(\N) \to F(\N)$ defined by
\[
(\La f)(x) = \sum_{n \in \N} w_k(n) f(n) M(n,x)
\]
extends to a unitary operator $\La:\mathcal H_k \to \mathcal H_k$, and intertwines $\pi_k$ with $\pi_k\circ \te_{-\sqrt{c}}$. Furthermore, the kernel $M(n,x)$ satisfies
\[
[\pi_k(X^*) M(\cdot,x)](n) = [\pi_k(\te_{-\sqrt{c}}(X))M(n,\cdot)](x), \qquad X \in U(\su(1,1)).
\]
\end{prop}
\begin{proof}
Unitarity follows from the orthogonality relations and completeness of the Meixner polynomials. The properties of the kernel follow from Lemma \ref{eq:actions on M(n,x)}.
\end{proof}

For $j=1,\ldots,N$ assume $k_j>0$ . From Proposition \ref{prop:intertwiner} we find unitary equivalence for tensor product representations,
\[
\pi_{k_1} \tensor \cdots \tensor \pi_{k_N}  \simeq (\pi_{k_1} \tensor \cdots \tensor \pi_{k_N}) \circ (\te_{-\sqrt{c}} \tensor \cdots \tensor \te_{-\sqrt{c}}).
\]
Using $\te_{-\sqrt{c}}\tensor \te_{-\sqrt{c}} \circ \De = \De \circ \te_{-\sqrt{c}}$ and $\te_{-\sqrt{c}}(\Om) = \Om$, we see that, as in Lemma \ref{lem:LSIP su11}, the right-hand-side applied to $\sum Y_{i,j}-2k_ik_j$ is the generator $L^{\mathrm{SIP}}$. Applying Theorem \ref{thm:duality} then leads to self-duality of $L^{\mathrm{SIP}}$ \eqref{eq:L SIP}, i.e.~self-duality of the symmetric inclusion process SIP($k$).
\begin{theorem} \label{thm:selfduality SIP}
The operator $L^{\mathrm{SIP}}$ defined by \eqref{eq:L SIP} is self-dual, with duality function
\[
\prod_{j=1}^N M(n_j,x_j;k_j,c).
\]
\end{theorem}

\begin{remark}
The Lie algebra $\su(2)$ is $\mathfrak{sl}(2,\C)$ equipped with the $*$-structure defined by $H^*=H, E^*=F$.
It is well known that $\su(2)$ has only finite dimensional irreducible $*$-representations. These can formally be obtained from the $\su(1,1)$ discrete series representation \eqref{eq:actions HEF} by setting $k=-j/2$ for some $j \in \N$, where $j+1$ is the dimension of the corresponding representation space. If we make the corresponding substitution $k_i=-j_i/2$ in the generator \eqref{eq:L SIP} of the symmetric inclusion process, we obtain the generator of the symmetric exclusion process SEP on $N$ sites where site $i$ can have at most $j_i$ particles. Making a similar substitution in Theorem \ref{thm:selfduality SIP} we find self-duality of SEP, with duality function given by a product of Krawchouck polynomials.
\end{remark}

\subsection{Laguerre polynomials and duality between SIP and BEP}
The Laguerre polynomials are defined by
\[
L_n^{(\al)}(x) = \frac{(\al+1)_n }{n!} \rFs{1}{1}{-n}{\al+1}{x}.
\]
They form an orthogonal basis for $L^2([0,\infty), x^\al e^{-x}dx)$, with orthogonality relations given by
\[
\int_0^\infty L_m^{(\al)}(x) L_n^{(\al)}(x) x^\al e^{-x} dx = \de_{mn} \frac{ \Ga(\al+n+1)}{n!}, \qquad \al>-1.
\]
The three-term recurrence relation is
\[
-x L_n^{(\al)}(x) = (n+1) L_{n+1}^{(\al)}(x) - (2n+\al+1) L_n^{(\al)}(x) + (n+\al) L_{n-1}^{(\al)}(x),
\]
and the differential equation is
\[
x \frac{d^2}{dx^2} L_n^{(\al)}(x) + (\al+1-x) \frac{d}{dx} L_n^{(\al)}(x) = -n L_n^{(\al)}(x).
\]
\*

We consider the action of the parabolic Lie algebra element $X_1  = -H+E-F$,
\[
[\pi_k(X_1) f](n) = -2(n+k) f(n) + \frac{n}{\sqrt{c}} f(n-1)+ \sqrt{c}(2k+n) f(n+1).
\]
Using the three-term recurrence relation for the Laguerre polynomials, we find the following result.
\begin{lemma} \label{lem:eigenfunctions X1}
The Laguerre polynomials $L(n,x;k)=\frac{n! c^{-\frac{n}{2}}}{(2k)_n}L_n^{(2k-1)}(x)$ are eigenfunctions of $\pi_k(X_1)$,
\[
[\pi_k(X_1) L(\cdot,x)](n) = -x L(n,x), \qquad x \in [0,\infty).
\]
\end{lemma}

Just as we did in the elliptic case, we can define an algebra isomorphism that will be useful. In this case, the element $X_1$ corresponds to the generator $E$.
\begin{lemma}
The assignments
\[
\begin{split}
\te(H)=E+F, \qquad \te(E)=\frac{i}{2}(-H+E-F), \qquad \te(F)=\frac{i}{2}(H+E-F),
\end{split}
\]
extend to a Lie algebra isomorphism $\te:\mathfrak{sl}(2,\C) \to \mathfrak{sl}(2,\C)$. Furthermore, $\te(\Om)=\Om$.
\end{lemma}
Note that $\te(E)=\frac{i}{2}X_1$. Furthermore, $\te$ does not preserve the $\su(1,1)$-$*$-structure, i.e.~$\te(X^*) \neq \te(X)^*$ in general. However, we can define another $*$-structure on $\mathfrak{sl}(2,\C)$ by $\star = \te^{-1} \circ * \circ \te$, then
\begin{equation} \label{eq:star-structure}
H^\star = -H, \qquad E^\star = - E, \qquad F^\star = - F,
\end{equation}
which is the $*$-structure of $i\mathfrak{sl}(2,\R)$.
Next we determine the actions of the generators on the Laguerre polynomials.
\begin{lemma}\label{lem:actions on Laguerre}
The Laguerre polynomials $L(n,x)$ satisfy
\[
\begin{split}
[\pi_k(\te(H)) L(\cdot,x)](n) &= 2x \frac{\partial}{\partial x} L(n,x) + (2k-x)\, L(n,x),\\
[\pi_k(\te(E)) L(\cdot,x)](n) &= -\frac12 ix\, L(n,x),\\
[\pi_k(\te(F)) L(\cdot,x)](n) &= -2ix \frac{\partial^2}{\partial x^2} L(n,x) - 2i(2k-x) \frac{\partial}{\partial x} L(n,x) +\frac{i}{2}(4k-x) \,L(n,x).
\end{split}
\]
\end{lemma}
\begin{proof}
The action of $\te(E)$ is Lemma \ref{lem:eigenfunctions X1}. From the differential equation for Laguerre polynomials we find
\[
[\pi_k(H)L(\cdot,x)](n) = 2(k+n) L(n,x) = -2x \frac{\partial^2}{\partial x^2} L(n,x) - 2(2k-x) \frac{\partial}{\partial x} L(n,x) +2k \,L(n,x).
\]
By linearity $\pi_k(H)$ extends to a differential operator acting on polynomials.
Then the action of $\te(H)$ is obtained from the identity $\te(H)=E+F= -\frac12[X_1,H]$. Finally, the action of $\te(F)$ follows from $\te(F) = \te(E) + iH$.
\end{proof}

Next we define an unbounded representation $\si_k$ of $\mathfrak{sl}(2,\C)$ on $\mathrm H_k = L^2([0,\infty), w(x;k)dx)$, where
\[
w(x;k) = \frac{ x^{2k-1}e^{-x} }{\Ga(2k)}.
\]
As a dense domain we take the set of polynomials $\mathcal P$. The representation $\si_k$ is defined on the generators $H, E, F$ by
\begin{equation} \label{eq:representation sigma}
\begin{split}
[\si_k(H) f](x) & = -2x \frac{\partial}{\partial x}f(x) - (2k-x) f(x), \\
[\si_k(E) f](x) & = -\frac12ix f(x),\\
[\si_k(F) f](x) & = -2ix \frac{\partial^2}{\partial x^2} f(x) - 2i(2k-x) \frac{\partial}{\partial x} f(x) +\frac{i}{2}(4k-x) f(x).
\end{split}
\end{equation}
Note that this is not a $*$-representation of $\su(1,1)$, but $\si_k \circ \te^{-1}$ is. Equivalently, $\si_k$ is a $*$-representation on $\mathrm H_k$ with respect to the $*$-structure defined by \eqref{eq:star-structure}. In the following lemma we give the intertwiner between $\pi_k$ and $\si_k\circ \te^{-1}$. The proof uses Lemma \ref{lem:actions on Laguerre} and orthogonality and completeness of the Laguerre polynomials.
\begin{prop} \label{prop:intertwiner Laguerre}
The operator $\La:F_0(\N) \to F([0,\infty))$ defined by
\[
(\La f)(x) = \sum_{n \in \N} w_k(n) f(n) L(n,x),
\]
extends to a unitary operator $\La:\mathcal H_k \to \mathrm H_k$ intertwining $\pi_k$ with $\si_k\circ \te^{-1}$. Furthermore, the kernel $L(n,x)$ satisfies
\[
[\pi_k(X^*) L(\cdot,x)](n) = [ \si_k(\te^{-1}(X)) L(n,\cdot)](x), \qquad X \in U(\su(1,1)).
\]
\end{prop}
For $j=1,\ldots,N$ let $k_j>0$, and define
\begin{equation} \label{eq:tensor product si}
\si_k = (\si_{k_1} \tensor \cdots \tensor \si_{k_N})\circ (\te^{-1} \tensor \cdots \tensor \te^{-1}),
\end{equation}
which is a $*$-representation of $\su(1,1)$ on $\bigotimes_{j=1}^n H_{k_j}$. The counterpart of Lemma \ref{lem:LSIP su11} for the representation $\si_k$ is as follows.
\begin{lemma} \label{lem:si(Y)=LBEP}
The generator $L^{\mathrm{BEP}}$ given by \eqref{eq:L BEP} satisfies
\[
L^{\mathrm{BEP}} = \sum_{1 \leq i< j \leq N} \si_k(Y_{i,j}) + 2k_ik_j,
\]
where $Y$ is given by \eqref{eq:Y}.
\end{lemma}
\begin{proof}
Using $\te \tensor \te \circ \De = \De \circ \te$ and $\te(\Om)=\Om$, we see that $(\te\tensor \te)(Y)=Y$. It is enough to calculate $\si_{k_1} \tensor \si_{k_2}(Y)$. Using \eqref{eq:Y} and \eqref{eq:representation sigma} we find
\[
\begin{split}
[\si_{k_1} \tensor \si_{k_2}(Y)f](x_1,x_2) = & -2k_1k_2  f(x_1,x_2)-2(x_1k_2-x_2k_1)\Big(\frac{\partial}{\partial x_1} - \frac{\partial}{\partial x_2} \Big) f(x_1,x_2) \\
& + x_1x_2 \Big(\frac{\partial^2}{\partial x_1^2} -2\frac{\partial^2}{\partial x_1\partial x_2} + \frac{\partial^2}{\partial x_2^2} \Big)f(x_1,x_2),
\end{split}
\]
which corresponds to the term with $(i,j) = (1,2)$ in \eqref{eq:L BEP}.
\end{proof}
Finally, application of Theorem \ref{thm:duality} gives duality between the symmetric inclusion process SIP($k$) and the Brownian energy process BEP($k$).
\begin{theorem} \label{thm:SIP - BEP}
The operators $L^{\mathrm{SIP}}$ defined by \eqref{eq:L SIP} and $L^{\mathrm{BEP}}$ defined by \eqref{eq:L BEP} are in duality, with duality function
\[
\prod_{j=1}^N L(n_j,x_j;k_j).
\]
\end{theorem}

\subsection{Bessel functions and self-duality of BEP}
The Bessel function of the first kind, see e.g.~\cite[Chapter 4]{AAR}, is defined by
\[
J_\nu(x) = \frac{ (x/2)^\nu }{\Ga(\nu+1)} \rFs{0}{1}{\mhyphen}{\nu+1}{-\frac{x^2}{4}}.
\]
The function $J_\nu(xy)$ is an eigenfunction of a second-order differential operator;
\[
T=-\frac{\partial^2}{\partial x^2} - \frac1x \frac{\partial}{\partial x} + \frac{\nu^2}{x^2} , \qquad TJ_\nu(xy) = y^2 J_\nu(xy).
\]
The Hankel transform is a unitary operator $\mathcal F_\nu: L^2([0,\infty),x dx) \to L^2([0,\infty),ydy)$ defined by
\[
(\mathcal F_\nu f)(y) = \int_0^\infty f(x) J_\nu(xy) x\, dx, \qquad \nu >-1,
\]
for suitable functions $f$, and the inverse is given by $\mathcal F_\nu^{-1} = \mathcal F_\nu$.\\

Let $k>0$. We consider the second-order differential operator $\si_k(F)$, see \eqref{eq:representation sigma}. Using the differential equation for the Bessel functions, we can find eigenfunctions of $\si_k(F)$ in terms of Bessel functions. We can also determine the actions of $H$ and $E$ on the eigenfunctions.
\begin{lemma} \label{lem:action on Bessel}
The Bessel functions $J(x,y;k) = e^{\frac12(x+y)} (xy)^{-k+\frac12} J_{2k-1}(\sqrt{xy})$ satisfy
\[
\begin{split}
[\si_k(H)J(\cdot,y)](x) &= -2y \frac{\partial}{\partial y}J(x,y) - (2k-y)J(x,y),\\
[\si_k(E)J(\cdot,y)](x) &=  2iy \frac{\partial^2}{\partial y^2} J(x,y) + 2i(2k-y) \frac{\partial}{\partial y} J(x,y) -\frac{i}{2}(4k-y) \,J(x,y),\\
[\si_k(F) J(\cdot,y)](x) & = \frac{1}{2}iy\, J(x,y).
\end{split}
\]
\end{lemma}
\begin{proof}
The action of $F$ follows from the differential equation for the Bessel functions. We have
\[
[\si_k(E)J(\cdot,y)](x)=-\frac{ix}{2} J(x,y),
\]
then using the self-duality of the Bessel functions, i.e.~symmetry in $x$ and $y$, we obtain the action of $E$. Finally, having the actions of $E$ and $F$, we find the action of $H$ from $H=[E,F]$.
\end{proof}
Using the Hankel transform we can now define a unitary intertwiner with a kernel that has the desired properties.
\begin{prop} \label{prop:intertwiner Bessel}
The operator $\La:\mathcal P \to F([0,\infty))$ defined by
\[
(\La f)(y) = \int_0^\infty f(x) J(x,y) w(x;k)\, dx,
\]
extends to a unitary operator $\La:\mathrm H_k \to \mathrm H_k$ intertwining $\si_k$ with itself. Furthermore, the kernel satisfies
\[
[\si_k(X^\star) J(\cdot,y)](x) = [\si_k(X) J(x,\cdot)](y).
\]
\end{prop}
\begin{proof}
Since the set of polynomials $\mathcal P$ is dense in $\mathrm H_k$, it is enough to define $\La$ on $\mathcal P$. Unitarity of $\La$ is essentially unitarity of the Hankel transform $\mathcal F_{2k-1}$. The intertwining property follows directly from Lemma \ref{lem:action on Bessel}.
\end{proof}
Note that $\La$ intertwines between $*$-representation with respect to the $*$-structure given by \eqref{eq:star-structure}. Equivalently, $\La$ intertwines $\si\circ \te^{-1}$ with itself, which is a $*$-representation with respect to the $\su(1,1)$-$*$-structure.

Let $k \in \R_{>0}^N$, and consider the tensor product representation $\si_k$ defined by \eqref{eq:tensor product si}. Then from Proposition \ref{prop:intertwiner Bessel}, Lemma \ref{lem:si(Y)=LBEP} and Theorem \ref{thm:duality} we obtain self-duality of the Brownian energy process BEP($k$).
\begin{theorem}
The operator $L^{\mathrm{BEP}}$ given by \eqref{eq:L BEP} is self-dual, with duality function given by
\[
\prod_{j=1}^N J(x_j,y_j;k_j).
\]
\end{theorem}

\subsection{More duality relations}
The Meixner polynomials from \S\ref{ssec:Meixner} can be considered as overlap coefficients between eigenvectors of the elliptic Lie algebra element $H$ and another elliptic element $X_{a(c)}$. There is a similar interpretation as overlap coefficients for the Laguerre polynomials (elliptic $H$ - parabolic $X_1$) and Bessel functions (parabolic $X_1$ - parabolic $X_{-1}$). So far we did not consider overlap coefficients involving a hyperbolic Lie algebra element, because there does not seem to be an interpretation in this setting for the element $Y$ from \eqref{eq:Y} as generator for a Markov process. However, the construction we used still works and leads to duality as operators between $L^{\mathrm{SIP}}$ or $L^{\mathrm{BEP}}$ and a difference operator $L^{\mathrm{hyp}}$ defined below, which may be of interest. We will give the main ingredients for duality between $L^{\mathrm{SIP}}$ and $L^{\mathrm{hyp}}$ in case $N=2$ using overlap coefficients between elliptic and hyperbolic bases, which can be given in terms of Meixner-Pollaczek polynomials, see also \cite{KJ1,GKR}.\\

The Meixner-Pollaczek polynomials are defined by
\[
P_n^{(\la)}(x;\phi) = e^{in\phi}\frac{ (2\la)_n}{n!} \rFs{2}{1}{-n,\la+ix}{2\la}{1-e^{-2i\phi}}.
\]
The orthogonality relations are
\[
\frac{1}{2\pi}\int_{-\infty}^\infty P_m(x)P_n(x)\, e^{(2\phi-\pi)x} |\Ga(\la+ix)|^2 \, dx = \de_{mn} \frac{ \Ga(n+2\la) }{(2\sin\phi)^{2\la} \, n!}, \qquad \la>0,\ 0 < \phi < \pi,
\]
and the Meixner-Pollaczek polynomials form an orthogonal basis for the corresponding weighted $L^2$-space.
The three-term recurrence relations is
\[
2x \sin\phi \, P_n(x) = (n+1)P_{n+1}(x) - 2(n+\la) \cos\phi\, P_n(x) + (n+2\la-1) P_{n-1}(x),
\]
and the difference equation is
\[
2(n+\la)\sin\phi \, P_n(x) = -ie^{i\phi} (\la-ix)\, P_n(x+i) + 2 x \cos\phi\, P_n(x) + ie^{-i\phi} (\la+ix) P_n(x-i).
\]

We also need the representation $\rho_k$ on $\mathrm{H}_k^\phi=L^2(\R, w_k^\phi(x)dx)$, with weight function
\[
w_k^\phi(x) = \frac{(2\sin\phi)^{2k} }{2\pi \Ga(2k)} e^{-\pi x} |\Ga(k+ix)|^2,
\]
given by
\[
\begin{split}
[\rho_k(H) f](x) &= 2ix f(x),\\
[\rho_k(E) f](x) &= (k-ix) f(x+i),\\
[\rho_k(F) f](x) & =-(k+ix) f(x-i).
\end{split}
\]
Now define an operator $L^{\mathrm{hyp}}$ on an appropriate dense subspace of $\mathrm{H}_{k_1}^{\phi} \tensor \mathrm{H}_{k_2}^\phi$ by
\[
L^{\mathrm{hyp}} = \rho_{k_1} \tensor \rho_{k_2}(Y) + k_1k_2,
\]
where $Y$ is given by \eqref{eq:Y}, then we see that $L^{\mathrm{hyp}}$ is the difference operator given by
\[
\begin{split}
[L^{\mathrm{hyp}} f](x_1,x_2) &= 2(k_1-ix_1)(k_2+ix_2) \Big( f(x_1+i,x_2-i) - f(x_1,x_2) \Big) \\
& \quad + 2(k_1+ix_1)(k_2-ix_2) \Big( f(x_1-i,x_2+i) - f(x_1,x_2)\Big).
\end{split}
\]
In order to obtain a duality relation, we use the Lie algebra isomorphism $\te_\phi: \mathfrak{sl}(2,\C) \to \mathfrak{sl}(2,\C)$ defined by
\[
\begin{split}
\te_\phi(H) &= \frac{i}{\sin\phi}\Big( -\cos \phi H +  E -  F\Big),\\
\te_\phi(E) & = \frac{1}{2i\sin\phi}\Big(- H + e^{-i\phi} E - e^{i\phi}F\Big),\\
\te_\phi(F) & = \frac{1}{2i\sin\phi}\Big( - H + e^{i\phi} E - e^{-i\phi}F\Big).
\end{split}
\]
Note that $\te_\phi(H) = \frac{i}{\sin\phi}X_{\cos\phi}$, see \eqref{eq:Xa}, which is a hyperbolic Lie algebra element.
The isomorphism does not preserve the $\su(1,1)$-$*$-structure, but we have $\te_\phi(X^*) = \te_\phi(X)^\star$, see \eqref{eq:star-structure}.
 Now consider the functions
\[
P(n,x;k,\phi) = e^{x\phi}\frac{n!}{(2k)_n} P^{(k)}_n(x;\phi).
\]
Using $H= \frac{-i}{\sin \phi} ( \cos\phi (\te_\phi(H)) - \te_\phi(E) - \te_\phi(F))$, the three-term recurrence relation and the difference equation one finds
\[
\begin{split}
[\pi_k(\te_\phi(H)) P(\cdot,x)](n) & = 2ix\, P(n,x),\\
[\pi_k(\te_\phi(E)) P(\cdot,x)](n) &= -(k-ix) P(n,x+i),\\
[\pi_k(\te_\phi(F)) P(\cdot,x)](n) & = (k+ix) P(n,x-i).
\end{split}
\]
so that
\[
[\pi_k(X^*) P(\cdot,x)](n) = [\rho_k(\te_\phi^{-1}(X)) P(n,\cdot)](x).
\]
Then we can construct a unitary intertwiner between $\pi_k\circ \te_\phi$ and $\rho_k$ with $P(x,n)$ as a kernel, but we do not actually need the intertwiner, since the kernel is enough to state the duality result. Using $\te_\phi(\Om)=\Om$ we obtain
duality between the operators $L^{\mathrm{SIP}}$ and $L^{\mathrm{hyp}}$, with duality function given by the product
\[
P(n_1,x_1;k_1,\phi) P(n_2,x_2;k_2,\phi).
\]

In a similar way we can find duality between $L^{\mathrm{BEP}}$ and $L^{\mathrm{hyp}}$ in terms of Laguerre functions, and also self-duality for $L^{\mathrm{hyp}}$ in terms of Meixner-Pollaczek functions.


\begin{thebibliography}{99}


\bibitem{AAR} G.E. Andrews, R. Askey, R. Roy, \textit{Special Functions}, Encycl. Math. Appl. 71, Cambridge Univ. Press, 1999.

\bibitem{BK} D.~Basu, K.B.~Wolf, \textit{The unitary irreducible representations of $\mathrm{SL}(2,\R)$ in all subgroup reductions}, J.~Math.~Phys.~\textbf{23} (1982), no.~2, 189--205.

\bibitem{CGGR}  G.~Carinci, C.~Giardin\`a, C.~Giberti, F.~Redig, \textit{Dualities in population genetics: a fresh look with new dualities}, Stochastic Process.~Appl.~\textbf{125} (2015), no.~3, 941--969.

\bibitem{C} I.~Corwin, \textit{Two ways to solve ASEP}, Topics in percolative and disordered systems, 1--13, Springer Proc.~Math.~Stat., \textbf{69}, Springer, New York, 2014.

\bibitem{FG} C.~Franceschini, C.~Giardin\`a, \textit{Stochastic duality and orthogonal polynomials}, \textsf{arXiv:1701.09115 [math.PR]}.

\bibitem{GKR} C.~Giardin\`a, J.~Kurchan, F.~Redig, \textit{Duality and exact correlations for a model of heat conduction}, J.~Math.~Phys.~\textbf{48} (2007), no.~3, 033301.

\bibitem{GKRV} C.~Giardin\`a, J.~Kurchan, F.~Redig, K.~Vafayi, \textit{Duality and hidden symmetries in interacting particle systems}, J.~Stat.~Phys.~\textbf{135} (2009), no.~1, 25--55.

\bibitem{GK}  W.~Groenevelt, E.~Koelink, \textit{Meixner functions and polynomials related to Lie algebra representations}, J.~Phys.~A \textbf{35} (2002), no.~1, 65--85.

\bibitem{GKR}  W.~Groenevelt, E.~Koelink, H.~Rosengren, \textit{Continuous Hahn functions as Clebsch-Gordan coefficients}, Theory and applications of special functions, 221--284, Dev.~Math., \textbf{13}, Springer, New York, 2005.

\bibitem{KLS}  R.~Koekoek, P.A.~Lesky, R.~Swarttouw, \textit{Hypergeometric orthogonal polynomials and their q-analogues}, Springer Monographs in Mathematics. Springer-Verlag, Berlin, 2010.

\bibitem{KJ1} H.T.~Koelink, J.~Van der Jeugt, \textit{Convolutions for orthogonal polynomials from Lie and quantum algebra representations}, SIAM J. Math. Anal. \textbf{29} (1998), 794--822.

\bibitem{KMP} C.~Kipnis, C.~Marchioro, E.~Presutti, \textit{Heat flow in an exactly solvable model},
J.~Statist.~Phys.~\textbf{27} (1982), no.1, 65--74.

\bibitem{K} T.H.~Koornwinder, \textit{Group theoretic interpretations of Askey's scheme of hypergeometric orthogonal polynomials}. Orthogonal polynomials and their applications (Segovia, 1986), 46--72,
Lecture Notes in Math., \textbf{1329}, Springer, Berlin, 1988.

\bibitem{L}  T.M.~Liggett, \textit{Interacting particle systems}, Reprint of the 1985 original. Classics in Mathematics. Springer-Verlag, Berlin, 2005.

\bibitem{RS} F.~Redig, F.~Sau, \textit{Duality functions and stationary product measures}, \textsf{arXiv:1702.07237 [math.PR]}.

\bibitem{S} F.~Spitzer, \textit{Interaction of Markov processes}, Advances in Math.~\textbf{5}, 1970, 246--290.

\bibitem{Sp} H.~Spohn, \textit{Long range correlations for stochastic lattice gases in a nonequilibrium steady state}, J.~Phys.~A \textbf{16} (1983), no.~18, 4275--4291.

\end{thebibliography}
\end{document}